\documentclass[reqno,b5paper]{amsart}
\usepackage{amsmath,amssymb,amsbsy,amsfonts,amsthm,latexsym,amsopn,amstext,amsxtra,euscript,amscd,amsthm}

\newtheorem{lemma}{Lemma}
\newtheorem{theorem}{Theorem}

\newcommand{\Q}{\mathbb Q}
\newcommand{\C}{\mathbb C}

\newcommand{\acr}{\newline\indent}

\begin{document}

\title[Diophantine triples and $k$-generalized Fibonacci sequences]{Diophantine triples with values in $k$-generalized Fibonacci sequences}

\author[C. Fuchs \and C. Hutle \and F. Luca \and L. Szalay]{Clemens Fuchs* \and Christoph Hutle* \and Florian Luca** \and Laszlo Szalay***}

\address{\llap{*\,}University of Salzburg \acr Hellbrunner Str. 34/I \acr 5020 Salzburg \acr AUSTRIA}
\email{clemens.fuchs@sbg.ac.at, christoph.hutle@gmx.at}
\address{\llap{**\,}U.~Witwatersrand, SOUTH AFRICA}
\email{florian.luca@wits.ac.za}
\address{\llap{***\,}U.~West Hungary, HUNGARY}
\email{szalay.laszlo@emk.nyme.hu}

\thanks{C.F. and C.H. were supported by FWF (Austrian Science Fund) grant No. P24574 and by the Sparkling Science project EMMA grant No. SPA 05/172.}

\maketitle

\begin{abstract}
We show that if $k\ge 2$ is an integer and  $(F_n^{(k)})_{n\ge 0}$ is the sequence of $k$-generalized Fibonacci numbers, then there are only finitely many triples of positive integers $1<a<b<c$ such that $ab+1,~ac+1,~bc+1$ are all members of $\{F_n^{(k)}: n\ge 1\}$. This generalizes a previous result (cf. \cite{FuLuSz1}) where the statement for $k=3$ was proved. The result is ineffective since it is based on Schmidt's subspace theorem.
\end{abstract}

\section{Introduction}
\label{sec:1}

There are many papers in the literature concerning Diophantine $m$-tuples which are sets of $m$ distinct positive integers $\{a_1,\ldots,a_m\}$ such that $a_ia_j+1$ is a square for all $1\le i<j\le m$ (see \cite{Du}, for example). A variation of this classical problem is obtained if one changes the set of squares by some different subset of positive integers like $k$-powers for some fixed $k\ge 3$, or perfect powers, or primes, or members of some linearly recurrent sequence, etc. (see \cite{FuLuSz}, \cite{LuSz}, \cite{LuSz1}, \cite{SzZi}, \cite{ISz1}). In this paper, we study this problem with the set of values of $k$-generalized Fibonacci numbers for some integer $k\ge 2$. Recall that these numbers denoted by $F_n^{(k)}$ satisfy the recurrence \[F_{n+k}^{(k)}=F_{n+k-1}^{(k)}+\cdots+F_n^{(k)}\] and start with $0,0,\ldots,0$ ($k-1$ times) followed by $1$. Notationwise, we assume that $F_i^{(k)}=0$ for $i=-(k-2),~-(k-1),~\ldots,~0$, and $F_1^{(k)}=1$. For $k=2$ we obtain the Fibonacci and for $k=3$ the sequence of Tribonacci numbers. The result is the following:

\begin{theorem}
\label{thm:main}
Let $k\ge 2$ be fixed. Then there are only finitely many triples of positive integers $1<a<b<c$ such that
\begin{equation}
\label{eq:1}
ab+1=F_x^{(k)},\qquad ac+1=F_y^{(k)},\qquad bc+1=F_z^{(k)}
\end{equation}
hold for some integers $x,y,z$.
\end{theorem}

Our result generalizes the results obtained in \cite{LuSz}, \cite{GoLu} and \cite{FuLuSz1}, where this problem was treated for the cases $k=2$ and $k=3$. In \cite{LuSz} it was shown that there does not exist a triple of positive integers $a,b,c$ such that $ab+1,ac+1,bc+1$ are Fibonacci numbers. In \cite{GoLu} it was shown that there is no Tribonacci Diophantine quadruple, that is a set of four positive integers $\{a_1,a_2,a_3,a_4\}$ such that $a_ia_j+1$ is a member of the Tribonacci sequence ($3$-generalized Fibonacci sequence) for $1\le i<j\le 4$, and in \cite{FuLuSz1} it was proved that there are only finitely many Tribonacci Diophantine triples. In the current paper we prove the same result for all such triples having values in the sequence of $k$-generalized Fibonacci numbers.

For the proof of Theorem \ref{thm:main}, we proceed as follows. In Section \ref{sec:2}, we recall some properties of the $k$-generalized Fibonacci sequence $F_n^{(k)}$ which we will need and we prove two lemmata. The first lemma shows that any $k-1$ roots of the characteristic polynomial are multiplicatively independent. In the second lemma the greatest common divisor of $F_x^{(k)}-1$ and $F_y^{(k)}-1$ for $2<y<x$ is estimated. In Section \ref{sec:4}, we assume that the Theorem \ref{thm:main} is false and give, using the Subspace Theorem, a finite expansion of infinitely many solutions. In Section \ref{sec:5}, we use a parametrization lemma which is proved by using results about finiteness of the number of non-degenerate solutions to $S$-unit equations. Applying it to the finite expansion, this leads us to a condition on the leading coefficient, which turns out to be wrong. This contradiction is obtained by showing that a certain Diophantine equations has no solutions; this last Diophantine equation has been treated in particular cases in \cite{CiLu} and \cite{Mar}.

\section{Preliminaries}
\label{sec:2}

There are already many results in the literature about $(F_n^{(k)})_{n\ge 0}$. We will only use what we need, which are the following properties. The sequence $(F_n^{(k)})_{n\ge 0}$ is linearly recurrent of  characteristic polynomial
\[\Psi_k(X)=X^k-X^{k-1}-\cdots-X-1.\]
The polynomial $\Psi_k(X)$ is separable and irreducible in ${\mathbb Q}[X]$ and the Galois group thus acts transitively on the roots, which we denote by $\alpha_1,\ldots,\alpha_k$. If $k$ is even or prime, the Galois group is certainly $S_k$ (see \cite{Mar} for these statements). The polynomial $\Psi_k(X)$ has only one root, with out loss of generality assume that $\alpha_1>1$, which is outside the unit disk (formally, this root depends also on $k$, but in what follows we shall omit the dependence on $k$ on this and the other roots of $\Psi_k(X)$ in order to avoid notational clutter). Thus,
\[
\Psi_k(X)=\prod_{i=1}^k (X-\alpha_i),\qquad {\text{\rm where}}\qquad |\alpha_i|<1,\quad i=2,\ldots,k.
\]
Observe that $\alpha_1\alpha_2\cdots\alpha_k=(-1)^{k-1}$. Note also that
\[
\Psi_k(X)=X^k-\left(X^{k-1}+\cdots+1\right)=X^k-\frac{X^k-1}{X-1}=\frac{X^{k+1}-2X^k+1}{X-1},
\]
a representation which is sometimes useful. Furthermore,
\begin{equation}
\label{eq:2}
2-\frac{1}{k}<\alpha_1<2
\end{equation}
(see Lemma 3 in \cite{Dre}).  The Binet formula of $(F_n^{(k)})_{n\ge 0}$ is given by
\begin{equation}
\label{eq:Binet}
F_n^{(k)}=\sum_{i=1}^k f_i \alpha_i^n\qquad {\text{\rm for~all}}\qquad n\ge 0,
\end{equation}
where
\begin{equation}
\label{eq:coefficient}
f_i=\frac{(\alpha_i-1)\alpha_i^{-1}}{2+(k+1)(\alpha_i-2)}\qquad i=1,\ldots,k
\end{equation}
(see Theorem 1 in \cite{Dre}). We have
\begin{equation}
\label{eq:approx}
|F_n^{(k)}-f_1\alpha_1^k|<\frac{1}{2}\qquad {\text{\rm for~all}}\qquad n\ge 1
\end{equation}
and
\begin{equation}
\label{eq:f1approx}
f_1 < 1
\end{equation}
(see Theorem 2 in \cite{Dre}). We also need the fact that
\begin{equation}
\label{eq:size}
\alpha_1^{n-2}<F_n^{(k)}<\alpha_1^{n-1}
\end{equation}
(see \cite{Bra}).

Furthermore, the following property is of importance; it follows from the fact that there is no non-trivial multiplicative relation between the conjugates of a Pisot number (cf. \cite{Mignotte}). Since in our case it is rather easy to verify we shall present a proof of what we need.

\begin{lemma}\label{prop:multindep}
Each set of $k-1$ different roots (e.g. $\{\alpha_1, \ldots, \alpha_{k-1}\}$) is multiplicatively independent.
\end{lemma}
\begin{proof}
We shall prove the statement only for the set $\{\alpha_1,\ldots,\alpha_{k-1}\}$. The general statement follows easily by slightly changing the arguments below.

Let us denote by $\mathbb{L} := \mathbb{Q}(\alpha_1, \dots, \alpha_k)$ the splitting field of $\Psi_k$ over $\mathbb{Q}$ and by $\mathcal{O}_\mathbb{L}$ the ring of integers in $\mathbb{L}$. Note, that the roots $\alpha_1,\ldots,\alpha_k$ are certainly in the group of units $\mathcal{O}_\mathbb{L}^\times$ which follows from
\begin{displaymath}
\alpha_i^{-1} = (-1)^{k-1}(\alpha_1 \cdots \alpha_{i-1} \cdot \alpha_{i+1} \cdots \alpha_k)
\end{displaymath}
for $i=1,\ldots,k$. The extension $\mathbb{L}/\mathbb{Q}$ is Galois. We denote by $d$ its degree and by $G:=\operatorname{Gal}(\mathbb{L}/\mathbb{Q})$ the Galois group of $\mathbb{L}$ over $\mathbb{Q}$. Let $G=\{\sigma_1,\ldots,\sigma_d\}$. We consider the map $\lambda:\mathcal{O}_\mathbb{L}^\times\rightarrow \mathbb{R}^d$ defined by $ x \mapsto (\log \vert \sigma_1(x) \vert,\log\vert\sigma_2(x)\vert,\ldots,\log\vert\sigma_d(x)\vert)$. Observe that by the product formula we have for $x\in \mathbb{L}^\times$ that \[\prod_{v\in M_\mathbb{L}}\vert x\vert_v=1;\] and for $x\in\mathcal{O}_\mathbb{L}^\times$ this leads to \[\prod_{\substack{v \in \mathbb{L} \\ v | \infty}} \vert x \vert_v = 1\] since for every finite place $v$ we have $\vert x\vert_v=1$. This means that the image of $\lambda$ lies in the hyperplane defined by $X_1+\cdots+X_d=0$ of $\mathbb{R}^d$.

We will use the property that $\lambda$ is a homomorphism (see e.g. \cite{JN}). So if we can prove, that the $k-1$ vectors $\lambda(\alpha_1),\ldots,\lambda(\alpha_{k-1})$ are linearly independent, this will prove the statement.

Since $\Psi_k$ is irreducible over $\mathbb{Q}$, the Galois group $G$ of $\mathbb{L}$ over $\mathbb{Q}$ acts transitively on $\{\alpha_1,\ldots,\alpha_k\}$, i.e. for each $i = 1, \dots, k$ there is some Galois automorphism which sends $\alpha_i$ to $\alpha_1$; without loss of generality let $\sigma_1,\ldots,\sigma_k$ be such that $\sigma_i(\alpha_i) = \alpha_1$. Observe that $\sigma_i^{-1}(\alpha_1)=\{\alpha_i\}$ for $i=1,\ldots,k$. We have \begin{equation*}\begin{split}\lambda(\alpha_1)&=(\log\vert\alpha_1\vert,\log\vert\sigma_2(\alpha_1)\vert,\ldots,\log\vert\sigma_{k-1}(\alpha_1)\vert,\ldots),\\
\lambda(\alpha_2)&=(\log\vert\sigma_1(\alpha_2)\vert,\log\vert\alpha_1\vert,\ldots,\log\vert\sigma_{k-1}(\alpha_2)\vert,\ldots),\\
&\vdots\\
\lambda(\alpha_{k-1})&=(\log\vert\sigma_1(\alpha_{k-1})\vert,\log\vert\sigma_2(\alpha_{k-1})\vert,\ldots,\log\vert\alpha_{1}\vert,\ldots).
\end{split}\end{equation*}
We will show that the matrix $(\log\vert\sigma_i(\alpha_j)\vert)_{i=1,\ldots,k-1;j=1,\ldots,k-1}$ consisting of the first $k-1$ entries of these $k-1$ vectors has rank $k-1$ implying the statement of the lemma.

Observe now that $\alpha_1\cdots\alpha_k=(-1)^{k-1}$ and thus $\vert\sigma_i(\alpha_1\cdots\alpha_k)\vert=1$ for all $i=1,\ldots,k-1$. It follows that $\sum_{j=1}^k\log\vert\sigma_i(\alpha_j)\vert=0$ and thus \[\sum_{j=1}^{k-1}\log\vert\sigma_i(\alpha_j)\vert=-\log\vert\sigma_i(\alpha_k)\vert>0\] for $j=1,\ldots,k-1$ and all $i=1,\ldots,k-1$. Hence the transpose of the matrix above is strictly diagonal-dominant since the diagonal entries, which are all equal to $\log\vert\alpha_1\vert$, are positive, all other entries are negative, and each row-sum (in the transpose matrix) of all off-diagonal entries is in absolute value less than the corresponding diagonal entry. It follows that the matrix is regular, which is what we wanted to show.
\end{proof}

Finally, we prove the following result, which generalizes Proposition 1 in \cite{FuLuSz1}. Observe that the upper bound depends now on $k$.

\begin{lemma}
\label{lem:1}
Let $x>y\ge 3$. Then
\begin{equation}
\label{eq:gcd}
\gcd(F_x^{(k)}-1,F_y^{(k)}-1)<\alpha_1^{\frac{kx}{k+1}}.
\end{equation}
\end{lemma}

\begin{proof}
We may assume that $y\ge 4$, since for $y=3$, we get $F_y^{(k)}-1=1$, and there is nothing to prove. We put $d=\gcd(F_x^{(k)}-1,F_y^{(k)}-1)$. Let $\kappa$ be a constant to be determined later. If $y\le \kappa x+1$, then
\begin{equation}
\label{eq:ysmall}
d\le F_y^{(k)}-1<F_y^{(k)}<\alpha_1^{y-1}\le \alpha_1^{\kappa x}.
\end{equation}
From now on, we assume that $y>\kappa x+1$. Using \eqref{eq:Binet} and \eqref{eq:approx}, we write
\begin{equation}\label{eq:Fxandy}\begin{split}
F_x^{(k)} & =  f_1 \alpha_1^x+\zeta_x,\qquad |\zeta_x|<1/2,\\
F_y^{(k)} & =  f_1 \alpha_1^y+\zeta_y,\qquad |\zeta_y|<1/2.
\end{split}\end{equation}
We put ${\mathbb K}={\mathbb Q}(\alpha_1)$. We let $\lambda=x-y<(1-\kappa )x-1$, and note that
\[
d\mid (F_x^{(k)}-1)-\alpha_1^{\lambda} (F_y^{(k)}-1)\qquad {\text{\rm in}}\qquad {\mathbb K}.
\]
We write
\[
d\eta=(F_x^{(k)}-1)-\alpha_1^{\lambda} (F_y^{(k)}-1),
\]
where $\eta$ is some algebraic integer in ${\mathbb K}$. Note that the right-hand side above is not zero, for if it were, we would get $\alpha_1^{\lambda}=(F_x^{(k)}-1)/(F_y^{(k)}-1)\in {\mathbb Q}$, which is false for $\lambda>0$. We compute
norms from ${\mathbb  K}$ to ${\mathbb Q}$. Observe that
\begin{equation*}\begin{split}
 \left|(F_x^{(k)}-1)-\alpha_1^{\lambda}(F_y^{(k)}-1)\right|& =  \left|(f_1\alpha_1^x+\zeta_x-1)-\alpha_1^{\lambda}(f_1\alpha_1^{y}+\zeta_y-1)\right|\\
& =  \left|\alpha_1^{\lambda}(1-\zeta_y)-(1-\zeta_x)\right| \\
& \le   \frac{3}{2}\alpha_1^{\lambda}-\frac{1}{2}<\frac{3}{2}\alpha_1^{\lambda}<\alpha_1^{\lambda+1}\le \alpha_1^{(1-\kappa)x}.
\end{split}\end{equation*}
In the above, we used the fact that
\[
-1/2<\zeta_x,\zeta_y<1/2
\]
(see \eqref{eq:approx}) as well as \eqref{eq:2}. Further, let $\sigma_i$ be any Galois automorphism that maps $\alpha_1$ to $\alpha_i$. Then for $i\ge 2$, we have
\begin{equation*}\begin{split}
\left|\sigma_i\left((F_x^{(k)}-1)-\alpha_1^{\lambda} (F_y^{(k)}-1)\right)\right| & =  \left|(F_x^{(k)}-1)-\alpha_i^{\lambda} (F_y^{(k)}-1)\right|\\
& <  F_x^{(k)}-1+F_y^{(k)}-1<\alpha^{x-1}+\alpha^{y-1}-2\\
& <  \alpha^{x-1}\left(1+\alpha^{-1}\right)\le  \alpha^x.
\end{split}\end{equation*}
We then have
\begin{equation*}\begin{split}
d^k & \le  |N_{{\mathbb K}/{\mathbb Q}}(d\eta)|\\
& \le  \left| N_{{\mathbb K}/{\mathbb  Q}}((F_x^{(k)}-1)-\alpha_1^{\lambda} (F_y^{(k)}-1))\right|\\
& =  \left|\prod_{i=1}^k \sigma_i\left((F_x^{(k)}-1)-\alpha_1^{\lambda} (F_y^{(k)}-1)\right)\right|\\
& <  \alpha_1^{(1-\kappa)x} (\alpha_1^{x})^{k-1}=\alpha_1^{(k-\kappa)x}.
\end{split}\end{equation*}
Hence,
\begin{equation}
\label{eq:ylarge}
d\le \alpha_1^{(1-\kappa/k)x}.
\end{equation}
In order to balance \eqref{eq:ysmall} and \eqref{eq:ylarge}, we choose $\kappa$ such that $\kappa=1-\kappa/k$, giving $\kappa=k/(k+1)$, and the lemma is proved.
\end{proof}

\section{Parametrizing the solutions}
\label{sec:4}
In order to simplify the notation we shall from now onwards write $F_n$ instead of $F_n^{(k)}$; we still mean the $n$th $k$-generalized Fibonacci number. The arguments in this section follow the arguments from \cite{FuLuSz1}. We will show that if there are infinitely many solutions to \eqref{eq:1}, then all of them can be parametrized by finitely many expressions as given in \eqref{eq:c} for $c$ below.

We assume that there are infinitely many solutions to \eqref{eq:1}. Then, for each integer solution $(a,b,c)$, we have
\[
a = \sqrt{\frac{(F_x - 1)(F_y - 1)}{F_z - 1}}, \, b = \sqrt{\frac{(F_x - 1)(F_z - 1)}{F_y - 1}}, \, c = \sqrt{\frac{(F_y - 1)(F_z - 1)}{F_x - 1}}.
\]
From
\[
\alpha_1^{x+y-2} \geq F_x F_y > (F_x - 1)(F_y - 1) \geq F_z - 1 \geq \alpha_1^{z-2} - 1 > \alpha_1^{z-3}
\]
we see that $x+y > z - 1$ and thus $y \geq z/2$.
In order to get a similar correspondence for $x$ and $z$, we denote $d_1 := \gcd(F_y - 1, F_z - 1)$ and $d_2 := \gcd(F_x - 1, F_z - 1)$, such that $F_z - 1 \mid d_1 d_2$. Then we use Lemma \ref{lem:1} to obtain
\[
\alpha_1^{x-1} > F_x > F_x-1  \geq d_2 \geq \frac{F_z - 1}{d_1} \geq \frac{\alpha_1^{z-2} - 1}{\alpha_1^{\frac{kz}{k+1}}} > \alpha_1^{z -\frac{kz}{k+1} -3}
\]
and hence
\[
x > \left( 1 - \frac{k}{k+1} \right) z - 2,
\]
which we can write as $x > C_1 z$ for some small constant $C_1 < 1$ (depending only on $k$), when $z$ is sufficiently large.

Next, we do a Taylor series expansion for $c$ which was given by
\begin{equation}\label{eq:c-exp}
c = \sqrt{\frac{(F_y - 1)(F_z - 1)}{F_x - 1}}.
\end{equation}
Using the power sum representations of $F_x, F_y, F_z$, we get
\begin{equation*}\begin{split}
c = &\sqrt{f_1} \alpha_1^{(-x+y+z)/2} \\
& \cdot \left( 1 + (-1/f_1) \alpha_1^{-x} + (f_2/f_1) \alpha_2^x \alpha_1^{-x} + \dots + (f_k/f_1) \alpha_k^x \alpha_1^{-x}\right)^{-1/2} \\
& \cdot \left( 1 + (-1/f_1) \alpha_1^{-y} + (f_2/f_1) \alpha_2^y \alpha_1^{-y} + \dots + (f_k/f_1) \alpha_k^y \alpha_1^{-y} \right)^{1/2} \\
& \cdot \left( 1 + (-1/f_1) \alpha_1^{-z} + (f_2/f_1) \alpha_2^z \alpha_1^{-z} + \dots + (f_k/f_1) \alpha_k^z \alpha_1^{-z}\right)^{1/2}.
\end{split}\end{equation*}
We then use the binomial expansion to obtain
\begin{equation*}\begin{split}
\left( 1 \right.&\left.+ (-1/f_1) \alpha_1^{-x} + (f_2/f_1) \alpha_2^x \alpha_1^{-x} + \dots + (f_k/f_1) \alpha_k^x \alpha_1^{-x} \right)^{1/2} \\
&= \sum_{j=0}^T \binom{1/2}{j} \left((-1/f_1) \alpha_1^{-x} + (f_2/f_1) \alpha_2^x \alpha_1^{-x} + \dots + (f_k/f_1) \alpha_k^x \alpha_1^{-x} \right)^j \\
&\hspace*{1cm}+ \mathcal{O}(\alpha_1^{-(T+1) x}),
\end{split}\end{equation*}
where $\mathcal{O}$ has the usual meaning, using estimates from \cite{FuTi} and where $T$ is some index, which we will specify later. Since $x < z$ and $z < x/C_1$, the remainder term can also be written as $\mathcal{O}(\alpha_1^{-T \|x\|/C_1})$, where $\|x\|=\max\{x,y,z\}=z$. Doing the same for $y$ and $z$ likewise and multiplying those expression gives
\begin{equation}\label{eq:c-approx}
c=\sqrt{f_1} \alpha_1^{(-x+y+z)/2} \left( 1 + \sum_{j=1}^{n-1} d_j M_j\right) + \mathcal{O}(\alpha_1^{-T \|x\|/C_1}),
\end{equation}
where the integer $n$ depends only on $T$ and where $J$ is a finite set, $d_j$ are non-zero coefficients in ${\mathbb L}:={\mathbb Q}(\alpha_1,\ldots,\alpha_k)$, and $M_j$ is a monomial of the form
\[
M_j=\prod_{i=1}^k \alpha_i^{L_{i,j}({\bf x})},
\]
in which ${\bf x}=(x,y,z)$, and $L_{i,j}({\bf x})$ are linear forms in ${\bf x}\in {\mathbb R}^3$ with integer coefficients which are all non-negative if $i=2,\ldots,k$ and non-positive if $i=1$.
Note that each monomial $M_j$ is ``small", that is there exists a constant $\kappa > 0$ (which we can even choose independently of $k$), such that
\begin{equation}\label{eq:Mj}
|M_j| \leq e^{-\kappa x}\qquad {\text{\rm for ~all}}\qquad j\in J.
\end{equation}
This follows directly from
\begin{equation*}\begin{split}
|M_j| &= |\alpha_1|^{L_{1,j}({\bf x})} \cdot |\alpha_2|^{L_{2,j}({\bf x})} \cdots |\alpha_k|^{L_{k,j}({\bf x})} \\
&\leq (2 - 1/k)^{L_{1,j}({\bf x})} \cdot 1 \cdots 1 \\
&\leq (3/2)^{-x} \\
&\leq e^{-\kappa x}\qquad {\text{\rm for ~all}}\qquad j\in J.
\end{split}\end{equation*}

Our aim of this section is to apply a version of the Subspace Theorem given in \cite{Ev} to show that there is a finite expansion of $c$ involving terms as in \eqref{eq:c-approx}; the version we are going to use can also be found in Section 3 of \cite{Fu2}. For the set-up - in particular the notion of heights - we refer to the mentioned papers.

We work with the field ${\mathbb L}={\mathbb Q}(\alpha_1,\ldots,\alpha_k)$ and let $S$ be the finite set of infinite places (which are normalized so that the Product Formula holds, cf. \cite{Ev}). Observe that $\alpha_1,\ldots,\alpha_k$ are $S$-units. %, that are either infinite or in the set $\{ v \in M_\mathbb{L}: |\alpha_1|_v \neq 1 \vee \cdots \vee |\alpha_k|_v \neq 1)\}$.
According to whether $-x+y+z$ is even or odd, we set $\epsilon = 0$ or $\epsilon = 1$ respectively, such that \[\alpha_1^{(-x+y+z-\epsilon)/2} \in \mathbb{L}.\] By going to a still infinite subset of the solutions we may assume that $\epsilon$ is always either $0$ or $1$.

Using the fixed integer $n$ (depending on $T$) from above, we now define $n+1$ linearly independent linear forms in indeterminants $(C, Y_0, \dots, Y_n)$.
For the standard infinite place $\infty$ on $\C$, we set
\begin{equation}
l_{0, \infty}(C, Y_0, \dots, Y_{n-1}) := C - \sqrt{f_1 \alpha_1^\epsilon} Y_0 - \sqrt{f_1 \alpha_1^\epsilon} \sum_{j=1}^{n-1} d_{j} Y_j,
\end{equation}
where $\epsilon \in \{0,1\}$ is as explained above, and
\[
l_{i, \infty}(C, Y_0, \dots, Y_{n-1}) := Y_{i-1} \quad \textrm{for } i \in \{1, \dots, n\}.
\]
For all other places $v$ in $S$, we define
\[
l_{0, v} := C, \qquad l_{i, v} := Y_{i-1} \quad \textrm{ for } i = 1, \dots, n.
\]
We will show, that there is some $\delta > 0$, such that the inequality
\begin{equation}\label{eq:lf-ineq}
\prod_{v\in S}\prod_{i=0}^{n}\frac{\vert l_{i,v}({\bf
y})\vert_{v}}{\vert {\bf y}\vert_{v}} <\left(\prod_{v\in S}\vert
\det(l_{0,v},\ldots,l_{n,v})\vert_{v}\right)\cdot \mathcal{H}({\bf
y})^{-(n+1)-\delta}\,
\end{equation}
is satisfied for all vectors
\[
{\bf y} = (c, \alpha_1^{(-x+y+z-\epsilon)/2}, \alpha_1^{(-x+y+z-\epsilon)/2} M_1, \dots, \alpha_1^{(-x+y+z-\epsilon)/2} M_{n-1}).
\]
The use of the correct $\epsilon \in \{0,1\}$ guarantees, that these vectors are indeed in $\mathbb{L}^n$.

First notice, that the determinant in \eqref{eq:lf-ineq} equals $1$ for all places $v$. Thus \eqref{eq:lf-ineq} reduces to
\[
\prod_{v\in S}\prod_{i=0}^{n}\frac{\vert l_{i,v}({\bf
y})\vert_{v}}{\vert {\bf y}\vert_{v}} < \mathcal{H}({\bf
y})^{-(n+1)-\delta},
\]
and the double product on the left-hand side can be split up into
\[
\vert c - \sqrt{f_1 \alpha_1^\epsilon} y_0 - \sqrt{f_1 \alpha_1^\epsilon} \sum_{j=1}^{n-1} d_{j} y_j \vert_\infty \cdot \prod_{\substack{v \in M_{\mathbb{L}, \infty}, \\ v \neq \infty}} \vert c \vert_v \cdot \prod_{v \in S \backslash M_{\mathbb{L}, \infty}} \vert c \vert_v \cdot \prod_{j=1}^{n-1} \prod_{v \in S} \vert y_j \vert_v.
\]
Now notice that the last double product equals $1$ due to the Product Formula and that
\[
\prod_{v \in S \backslash M_{\mathbb{L}, \infty}} \vert c \vert_v \leq 1,
\]
since $c \in \mathbb{Z}$.
An upper bound on the number of infinite places in $\mathbb{L}$ is $k!$ and hence
\begin{equation*}\begin{split}
\prod_{\substack{v \in M_{\mathbb{L}, \infty}, \\ v \neq \infty}} \vert c \vert_v &< \left( \frac{(T_y - 1) (T_z - 1)}{T_x - 1} \right)^{k!} \\
&\leq (f_1 \alpha_1^y + \cdots + f_k \alpha_k^y - 1)^{k!} (f_1 \alpha_1^z + \cdots + f_k \alpha_k^z - 1)^{k!} \\
&\leq (1 \cdot 2^{\| x \|} + 1/2)^{2 \cdot k!}
\end{split}\end{equation*}
using \eqref{eq:f1approx} and \eqref{eq:approx}.
And finally the first expression is just
\[
\Big| \sqrt{f_1 \alpha_1^\epsilon}  \alpha^{(-x+y+z-\epsilon)/2} \sum_{j \geq n} d_j M_j \Big|,
\]
which, by \eqref{eq:c-approx}, is smaller than some expression of the form $C_2 \alpha_1^{T \|x\| / C_1}$.
Therefore we have
\[
\prod_{v\in S}\prod_{i=0}^{n}\frac{\vert l_{i,v}({\bf
y})\vert_{v}}{\vert {\bf y}\vert_{v}} < C_2 \alpha_1^{T \|x\| / C_1} \cdot (2^{\| x \|} + 1/2)^{2 \cdot k!}.
\]

Now we choose $T$ (and the corresponding $n$) in such a way that
\[C_2 \alpha_1^{-T\|x\|/C_1} < \alpha_1^{-\frac{T\|x\|}{2C_1}}\]
and
\[(2^{\|x\|} + 1/2)^{2 \cdot k!} < \alpha_1^{\frac{T\|x\|}{4C_1}}\]
holds. Then we can write
\begin{equation}
\prod_{v\in S}\prod_{i=0}^{n}\frac{\vert l_{i,v}({\bf y})\vert_{v}}{\vert {\bf y}\vert_{v}} < \alpha_1^{\frac{-T \|x\|}{2C_1}}.
\end{equation}

For the height of our vector ${\bf y}$, we estimate
\begin{equation*}\begin{split}
\mathcal{H}({\bf y}) &\leq C_3 \cdot \mathcal{H}(c) \cdot \mathcal{H}(\alpha_1^\frac{-x+y+z-\epsilon}{2})^n \cdot \prod_{i=1}^{n-1} \mathcal{H}(M_j) \\
&\leq C_3(2^{\|x\|} + 1/2)^{k!} \prod_{i=1}^{n-1} \alpha_1^{C_4 \|x\|} \\
&\leq \alpha_1^{C_5 \|x\|},
\end{split}\end{equation*}
with suitable constants $C_3, C_4, C_5$. For the second estimate, we used that
\[
\mathcal{H}(M_j) \leq \mathcal{H}(\alpha_1)^{C_{\alpha_1}({\bf x})}\mathcal{H}(\alpha_2)^{C_{\alpha_2}({\bf x})} \cdots \mathcal{H}(\alpha_k)^{C_{\alpha_k}({\bf x})}
\]
and bounded it by the maximum of those expressions. Furthermore we have
\[
\mathcal{H}(\alpha_1^\frac{-x+y+z-\epsilon}{2})^n \leq \alpha_1^{n \|x\|},
\]
which just changes our constant $C_4$.

Now finally, the estimate
\[
\alpha_1^{-\frac{T \|x\|}{2C_1}} \leq \alpha_1^{-\delta C_5\|x\|}
\]
is satisfied, when we pick $\delta$ small enough.

So all the conditions for the Subspace Theorem are met. Since we assumed that there are infinitely many solutions $(x,y,z)$ of \eqref{eq:lf-ineq}, we now can conclude, that all of them lie in finitely many proper linear subspaces. Therefore, there must be at least one proper linear subspace, which contains infinitely many solutions and we see that there exists a finite set $J$  and (new) coefficients $e_j$ for $j\in J$ in ${\mathbb L}$ such that we have
\begin{equation}\label{eq:c}
c=\alpha_1^{(-x+y+z-\epsilon)/2} \left(e_0+\sum_{j\in J_c} e_j M_j\right)
\end{equation}
with (new) non-zero coefficients $e_j$ and monomials $M_j$ as before.

Likewise, we can find finite expressions of this form for $a$ and $b$.

\section{Proof of the theorem}\label{sec:5}

We use the following parametrization lemma:

\begin{lemma}
Suppose, we have infinitely many solutions for \eqref{eq:1}. Then there exists a line in $\mathbb{R}^3$ given by
\begin{displaymath}
x(t) = r_1 t + s_1 \quad y(t) = r_2 t + s_2 \quad z(t) = r_3 t + s_3
\end{displaymath}
with rationals $r_1, r_2, r_3, s_1, s_2, s_3$, such that infinitely many of the solutions $(x,y,z)$ are of the form $(x(n), y(n), z(n))$ for some integer $n$.
\end{lemma}

\begin{proof}
Assume that \eqref{eq:1} has infinitely many solutions. We already deduced in Section \ref{sec:4} that $c$ can be written in the form
\begin{displaymath}
c=\alpha_1^{(-x+y+z-\epsilon)/2} \left(e_{c,0}+\sum_{j\in J_c} e_{c,j} M_{c,j}\right)
\end{displaymath}
with $J_c$ being a finite set, $e_{c,j}$ being coefficients in $\mathbb{L}$ for $j \in J_c \cup \{0\}$ and $M_{c,j} = \prod_{i=1}^k \alpha_i^{L_{c,i,j}({\bf x})}$ with ${\bf x}=(x,y,z)$.
In the same manner, we can write
\begin{displaymath}
b=\alpha_1^{(x-y+z-\epsilon)/2} \left(e_{b,0}+\sum_{j\in J_b} e_{b,j} M_{b,j}\right).
\end{displaymath}
Since $1 + bc = F_z = f_1 \alpha_1^z + \cdots + f_k \alpha_k^z$, we get
\begin{equation}\label{eq:s-uniteq}
f_1 \alpha_1^z + \cdots + f_k \alpha_k^z - \alpha_1^{z-\varepsilon}  \left(e_{b,0}+\sum_{j\in J_b} e_{b,j} M_{b,j}\right)\left(e_{c,0}+\sum_{j\in J_c} e_{c,j} M_{c,j}\right) = 1.
\end{equation}
Substituting
\begin{displaymath}
\alpha_k = \frac{(-1)^{k-1}}{\alpha_1 \cdots \alpha_{k-1}}
\end{displaymath}
into \eqref{eq:s-uniteq}, we obtain an equation of the form
\begin{equation}\label{eq:s-uniteq2}
\sum_{j \in  J} e_j \alpha_1^{L_{1,j}({\bf x})} \cdots \alpha_{k-1}^{L_{k-1,j}({\bf x})} = 0,
\end{equation}
where again $J$ is some finite set, $e_j$ are non-zero coefficients in $\mathbb{L}$ and $L_{i,j}$ are linear forms in ${\bf x}$ with integer coefficients.

This is an $S$-unit equation, where $S$ is the multiplicative group generated by $\{\alpha_1, \dots, \alpha_k, -1\}$. We may assume that infinitely many of the solutions ${\bf x}$ are non-degenerate solutions of (\ref{eq:s-uniteq2}) by replacing the equation by an equation given by a suitable vanishing subsum if necessary.

We may assume, that $(L_{1,i}, \dots, L_{k-1,i}) \neq (L_{1,j}, \dots, L_{k-1,j})$ for any $i \neq j$, because otherwise we could just merge these two terms.

Therefore for $i \neq j$, the theorem on non-degenerate solutions to $S$-unit equations (see \cite{ESS}) yields that the set of \[\alpha_1^{L_{1,i}({\bf x}) - L_{1,j}({\bf x})} \cdots \alpha_{k-1}^{L_{k-1,i}({\bf x}) - L_{k-1,j}({\bf x})}\] is contained in a finite set of numbers.
By Lemma \ref{prop:multindep}, $\alpha_1, \dots, \alpha_{k-1}$ are multiplicatively independent and thus the exponents $(L_{1,i} - L_{1,j})({\bf x}), \ldots, (L_{k-1,i} - L_{k-1,j})({\bf x})$ take the same value for infinitely many ${\bf x}$. Since we assumed, that these linear forms are not all identically zero, this implies, that there is some non-trivial linear form $L$ defined over $\Q$ and some $c\in\mathbb{Q}$ with $L({\bf x}) = c$ for infinitely many ${\bf x}$.
So there exist rationals $r_i, s_i, t_i$ for $i = 1, 2, 3$ such that we can parametrise
\begin{displaymath}
x = r_1 p + s_1 q + t_1, \quad y = r_2 p + s_2 q + t_2, \quad z = r_3 p + s_3 q + t_3
\end{displaymath}
with infinitely many pairs $(p,q) \in \mathbb{Z}^2$.

We can assume, that $r_i, s_i, t_i$ are all integers. If not, we define $\Delta$ as the least common multiple of the denominators of $r_i, s_i$ ($i= 1, 2,3$) and let $p_0, q_0$ be such that for infinitely many pairs $(p, q)$ we have $p \equiv p_0 \mod \Delta$ and $q \equiv q_0 \mod \Delta$. Then $p = p_0 + \Delta \lambda, q = q_0 + \Delta \mu$ and
\begin{equation*}\begin{split}
x & =  (r_1\Delta) \lambda+(s_1\Delta) \mu+(r_1p_0+s_1q_0+t_1)\\
y & =  (r_2\Delta) \lambda+(s_2\Delta) \mu+(r_2p_0+s_2q_0+t_2)\\
z & =  (r_3\Delta) \lambda+(s_3\Delta) \mu+(r_3p_0+s_3q_0+t_3).
\end{split}\end{equation*}
Since $r_i \Delta$, $s_i \Delta$ and $x, y, z$ are all integers, $r_i p_0 + s_i q_0 + t_i$ are integers as well. Replacing $r_i$ by $r_i \Delta$, $s_i$ by $s_i \Delta$ and $t_i$ by $r_i p_0 + s_i q_0 + t_i$, we can indeed assume, that all coefficients $r_i, s_i, t_i$ in our parametrization are integers.

Using a similar argument as in the beginning of the proof, we get that our equation is of the form
\begin{displaymath}
\sum_{j \in J} e'_j \alpha_1^{L'_{1,j}({\bf r})} \cdots \alpha_{k-1}^{L'_{k-1,j}({\bf r})} = 0,
\end{displaymath}
where ${\bf r} := (\lambda, \mu)$, $J$ is a finite set of indices, $e_j'$ are new non-zero coefficients in $\mathbb{L}$ and $L'_{i,j}({\bf r})$ are linear forms in ${\bf r}$ with integer coefficients.
Again we may assume that we have $(L'_{1,i}({\bf r}), \dots, L'_{k-1,i}({\bf r})) \neq (L'_{1,j}({\bf r}), \dots, L'_{k-1,j}({\bf r}))$ for any $i \neq j$.

Applying the theorem of non-degenerate solutions to $S$-unit equations once more, we obtain a finite set of numbers $\Lambda$, such that for some $i \neq j$, we have
\begin{displaymath}
\alpha_1^{(L'_{1,i} - L'_{1,j})({\bf r})} \cdots \alpha_{k-1}^{(L'_{k-1,i} - L'_{k-1,j})({\bf r})} \in \Lambda.
\end{displaymath}
So every ${\bf r}$ lies on a finite collection of lines and since we had infinitely many ${\bf r}$, there must be some line, which contains infinitely many solution, which proves our lemma.
\end{proof}

We apply this lemma and define $\Delta$ as the least common multiple of the denominators of $r_1, r_2, r_3$. Infinitely many of our $n$ will be in the same residue class modulo $\Delta$, which we shall call $r$.
Writing $n = m \Delta + r$, we get
\begin{displaymath}
(x,y,z) = ((r_1 \Delta) m + (r r_1 + s_1), (r_2 \Delta) m + (r r_2 + s_2), (r_3 \Delta) m + (r r_3 + s_3) ).
\end{displaymath}
Replacing $n$ by $m$, $r_i$ by $r_i \Delta$ and $s_i$ by $r r_i + s$, we can even assume, that $r_i, s_i$ are integers.
So we have
\begin{displaymath}
\frac{-x+y+z-\epsilon}{2} = \frac{(-r_1+r_2+r_3)m}{2} + \frac{-s_1+s_2+s_3 - \epsilon}{2}.
\end{displaymath}
This holds for infinitely many $m$, so we can choose a still infinite subset such that all of them are in the same residue class $\delta$ modulo $2$ and we can write $m = 2 \ell + \delta$ with fixed $\delta \in \{0,1\}$.
Thus we have
\begin{displaymath}
\frac{-x+y+z-\epsilon}{2} = (-r_1 + r_2 + r_3) \ell + \eta,
\end{displaymath}
where $\eta \in \mathbb{Z}$ or $\eta \in \mathbb{Z} + 1/2$.

Using this representation, we can write \eqref{eq:c} as
\begin{equation}\label{eq:c-par}
c(\ell) = \alpha_1^{(-r_1+r_2+r_3)\ell + S} \left(e_0+\sum_{j\in J_c} e_j M_j\right).
\end{equation}
for infinitely many $\ell$, where
\[
M_j=\prod_{i=1}^k \alpha_i^{L_{i,j}({\bf x})},
\]
as before and ${\bf x} = {\bf x}(\ell) = (x(2\ell + \delta), y(2\ell + \delta), z(2\ell + \delta))$.
From this, we will now derive a contradiction.

First we observe, that there are only finitely many solutions of \eqref{eq:c-par} with
$c(\ell) = 0$. That can be shown by using the fact, that a simple non-degenerate linear recurrence has only finite zero-multiplicity (see \cite{ESS} for an explicit bound).
We will apply this statement here for the linear recurrence in $\ell$; it only remains to check, that no quotient of two distinct roots of the form $\alpha_1^{L_{1,i}({\bf x}(\ell))} \cdots \alpha_k^{L_{k,i}({\bf x}(\ell))}$ is a root of unity or, in other words, that
\begin{equation}\label{eq:unit-eq}
(\alpha_1^{m_1} \alpha_2^{m_2} \cdots \alpha_k^{m_k})^n = 1
\end{equation}
has no solutions in $n \in \mathbb{Z}/\{0\}$, $m_1 < 0$ and $m_i > 0$ for $i = 2, \dots, k$. Assume relation \eqref{eq:unit-eq} holds. Replacing $\alpha_1$ by $(-1)^{k-1}(\alpha_2\cdots\alpha_k)^{-1}$ gives \[\alpha_2^{2(m_2-m_1)}\cdots\alpha_k^{2(m_k-m_1)}=1.\] By squaring this equation and applying Lemma \ref{prop:multindep} we get $2(m_2-m_1)=\cdots=2(m_k-m_1)=0$ and thus $m_1=m_2=\cdots=m_k$, which is impossible because of the signs of $m_1$ and $m_2,\ldots,m_k$.
%We apply the complex conjugation automorphism $\sigma$, that invaries $\alpha_1$ and switches complex conjugate pairs of some other roots $\alpha_i$.
%We get a new equation
%\begin{equation}\label{eq:unit-eq-conj}
%(\alpha_1^{m_1} \sigma(\alpha_2)^{m_2} \cdots \sigma(\alpha_k)^{m_k})^n = 1
%\end{equation}
%Multiplying \eqref{eq:unit-eq} and \eqref{eq:unit-eq-conj}, we get a new equation $A^n = 1$, where $A$ is a product consisting of the following factors:
%\begin{itemize}
%\item $\alpha_1^{2m_1}$, which is certainly real and strictly smaller than $1$;
%\item expressions of the form $(\alpha_i \alpha_j)^{m_i m_j}$, which come from pairs of complex conjugate roots $\alpha_i, \alpha_j$, these are reals with absolute value strictly smaller than $1$;
%\item possibly other real factors of the form $\alpha_\ell^{2m_\ell}$, if $\alpha_\ell$ is another real root, having again absolute value $<1$ though.
%\end{itemize}
%Altogether we have $A \in \mathbb{R}$ with $|A| < 1$, which contradicts our assumption, that $A^n$ should be $1$ for some $n \in \mathbb{Z}/\{0\}$.

So we have confirmed, that $c(\ell) \neq 0$ for still infinitely many solutions.
We use \eqref{eq:c-exp} and write
\begin{equation}\label{eq:xcyz}
(F_x - 1) c^2 = (F_y - 1)(F_z - 1).
\end{equation}
Then we insert the finite expansion \eqref{eq:c-par} in $\ell$ for $c$ into \eqref{eq:xcyz}.
Furthermore, we use the Binet formula \eqref{eq:Binet} and write $F_x, F_y, F_z$ as power sums in $x$, $y$ and $z$ respectively.
Using the parametrization $(x,y,z) = (r_1 m + s_1, r_2 m + s_2, r_3 m + s_3)$ with $m = 2\ell$ or $m = 2\ell+1$ as above, we have expansions in $\ell$ on both sides of \eqref{eq:Binet}. Since there must be infinitely many solutions in $\ell$, the largest terms on both sides have to grow with the same rate.
In order to find the largest terms, we have to distinguish some cases: If we assume, that $e_0 \neq 0$ for infinitely many of our solutions, then $e_0 \alpha_1^{(-x+y+z-\epsilon)/2}$ is the largest term in the expansion of $c$ and we have
\begin{displaymath}
f_1 \alpha_1^x e_0^2 \alpha_1^{-x+y+z-\epsilon} = f_1 \alpha_1^y f_1 \alpha_1^z.
\end{displaymath}
It follows that $e_0^2 = f_1 \alpha_1^{\epsilon}$. The case $e_0 = 0$ for infinitely many of our solutions
is not possible, because then, the right-hand side of \eqref{eq:xcyz} would grow faster than the left-hand side so that \eqref{eq:xcyz} could be true for only finitely many of our $\ell$.
In the other cases, we have $e_0 = \sqrt{f_1 \alpha_1^\epsilon}$, where $\epsilon \in \{0,1\}$. This now contradicts the following lemma, which turns out to be slightly more involved than in the special case on Tribonacci numbers (cf. \cite{FuLuSz1}).

\begin{lemma}
$\sqrt{f_1} \notin \mathbb{L}$ and $\sqrt{f_1\alpha_1} \notin \mathbb{L}$.
\end{lemma}

\begin{proof}
Suppose that $\sqrt{f_1\alpha_1^{\epsilon}}\in\mathbb{L}$ for some $\epsilon\in\{0,1\}$. Then there is $\beta\in\mathbb{L}$ such that $f_1\alpha_1^\epsilon=\beta^2$. Using \eqref{eq:coefficient}, we get that
\[
\frac{(\alpha_1-1)\alpha_1^{-\epsilon}}{2+(k+1)(\alpha_1-2)}=\beta^2.
\]
Computing norms over ${\mathbb Q}$, we get that
\begin{equation}
\label{eq:dio}
\left|\frac{N_{{\mathbb L}/{\mathbb Q}}(\alpha_1)^{-\epsilon} N_{{\mathbb L}/{\mathbb Q}}(\alpha_1-1)}{N_{{\mathbb L}/{\mathbb Q}}(2+(k+1)(\alpha_1-2))}\right|=
N_{{\mathbb L}/{\mathbb Q}}(\beta)^2=\square,
\end{equation}
where $\square$ denotes a rational square. Note that
\[
\left| N_{{\mathbb L}/{\mathbb Q}}(\alpha_1)\right|=\left| \prod_{i=1}^k \alpha_i\right|=|(-1)^k \cdot (-1)|=1,
\]
and
\[
\left|N_{{\mathbb L}/{\mathbb Q}}(\alpha_1-1)\right|=\left|\prod_{i=1}^k (\alpha_i-1)\right|=\left|\Psi_k(1)\right|=k-1,
\]
and finally that
\begin{equation*}\begin{split}
 \left|N_{{\mathbb L}/{\mathbb Q}}(2+(k+1)(\alpha_1-2))\right| & = \left|N_{{\mathbb L}/{\mathbb Q}}((k+1)\alpha_1-2k)\right|\\
& =  \left|\prod_{i=1}^k ((k+1)\alpha_i-2k)\right|\\
& =  (k+1)^{k} \left|\prod_{i=1}^k (2k/(k+1)-\alpha_i)\right|\\
& =  (k+1)^{k}  \left|\Psi_k(2k/(k+1))\right|\\
& =  (k+1)^k\left| \frac{X^{k+1}-2X^k+1}{X-1}\Big|_{X=2k/(k+1)}\right|\\
& =  \frac{2^{k+1} k^k-(k+1)^{k+1}}{k-1}.
\end{split}\end{equation*}
Hence, we get that equation \eqref{eq:dio} leads to
\[
\frac{2^{k+1} k^k-(k+1)^{k+1}}{(k-1)^2}=\square.
\]
This leads to
\begin{equation}
\label{eq:k}
2^{k+1} k^k-(k+1)^{k+1}=w^2
\end{equation}
for some integer $w$. But this equation has no integer solutions, which is proved in the theorem below. This concludes the proof.
\end{proof}

In order to finish the proof we have the following result, which might be of independent interest since particular cases were considered before in \cite{CiLu} and \cite{Mar}.

\begin{theorem}
The Diophantine equation \eqref{eq:k} has no positive integer solutions $(k,w)$ with $k\ge 2$.
\end{theorem}

\begin{proof}
The cases $k\equiv 1,2\pmod 4$ have already been treated both in \cite{CiLu} and in \cite{Mar}. We treat the remaining cases. If $k\equiv 0\pmod 4$, then the left-hand side of \eqref{eq:k}
is congruent to $-1\pmod 4$, and therefore it cannot be a square. Finally, assume that $k\equiv 3\pmod 4$. Then $k+1$ is even, $2^{(k+1)/2}\mid w$, and putting $w_1=w/2^{(k+1)/2}$, we get
\[
k^k-((k+1)/2)^{k+1}=w_1^2.
\]
We then get
\begin{equation}\label{eq:ktok}%\begin{split}
k^k  =  w_1^2+((k+1)/2)^{k+1}.
%\\& =  \left(w+i ((k+1)/2)^{((k+1)/2}\right)\cdot \left(w-i ((k+1)/2)^{(k+1)/2}\right),\end{split}
\end{equation}
Note that the two numbers in the right-hand side of \eqref{eq:ktok} are coprime, for if $p$ divides $w_1$ and $(k+1)/2$, then $p$ divides the left-hand side of \eqref{eq:ktok}. Thus $p$ divides both $k$ and $(k+1)/2$, so also $k-2((k+1)/2)=-1$, a contradiction. Thus, the right-hand side is a sum of two coprime squares and therefore all odd prime factors of it must be $1$ modulo $4$ contradicting the fact that in the left-hand side we have $k\equiv 3\pmod 4$. This finishes the proof of this theorem.
%and the two factors in the right-hand side of \eqref{eq:ktok} above are coprime in ${\mathbb Q}(i)$. Now the proof proceeds as on p. 32 of \cite{CiLu} to reduce the problem to the existence of a member of a Lucas sequence without a primitive divisor, which does not exist by work of Bilu, Hanrot and Voutier \cite{BHV}. This finishes the proof of this theorem, and of Theorem \ref{thm:main} as well.
\end{proof}


\begin{thebibliography}{99}

%\bibitem{BHV}  BILU, Yu.---HANRO, G.---VOUTIER, P. M.: \textit{Existence of primitive divisors of Lucas and Lehmer numbers. With an appendix by M. Mignotte}, J. Reine Angew. Math. {\bf 539} (2001), 75--122.

\bibitem{Bra} BRAVO, J. J.---LUCA, F.: \textit{On a conjecture about repdigits in k-generalized Fibonacci sequences}, Publ.~Math.~Debrecen \textbf{82} (2013),  623--639.

\bibitem{CiLu} CIPU, M.---LUCA, F.: \textit{On the Galois group of the generalized Fibonacci polynomial}, Ann. \c St. Univ. Ovidius Constan\c ta {\bf 9} (2001), 27--38.

\bibitem{Dre} DRESDEN, G. P.---DU, Z.: \textit{A Simplified Binet Formula for $k$-Generalized Fibonacci Numbers}, J. Integer Sequences {\bf 17} (2014), Article 14.4.7.

\bibitem{Du} DUJELLA, A.: \textit{There are only finitely many Diophantine quintuples}, J. Reine Angew. Math. {\bf 566} (2004), 183--214.

\bibitem{Ev}  EVERTSE, J.-H.: \textit{An improvement of the quantitative Subspace Theorem}, Compos. Math. \textbf{101} (1996), 225--311.

\bibitem{ESS}  {J. H.---EVERTSE, W. M.---SCHMIDT, H. P.---SCHLICKEWEI}.: \textit{Linear equations in variables which lie in a multiplicative group}, Annals of Mathematics \textbf{155.3} (2002), 807--836.

\bibitem{FuLuSz} FUCHS, C.---LUCA, F.---SZALAY, L.: \textit{Diophantine triples with values in binary recurrences}, Ann.~Sc.~ Norm.~Super.~Pisa~Cl.~Sc.~(5) \textbf{7} (2008), 579--608.

\bibitem{FuLuSz1} FUCHS, C.---HUTLE, C.--- IRMAK, N.---LUCA, F.---SZALAY, L.: \textit{Only finitely many Tribonacci Diophantine triples exist}, to appear in Math. Slovaca; arXiv:1508.07760.

\bibitem{FuTi} FUCHS, C.---TICHY, R.F.: \textit{Perfect powers in linear recurrence sequences}, Acta Arith. {\bf 107.1} (2003), 9--25.

\bibitem{Fu2} FUCHS, C.: \textit{Polynomial-exponential equations and linear recurrences}, Glas. Mat. \textbf{38(58)} (2003), no. 2, 233--252.

%\bibitem{Fu5} FUCHS, C.: \textit{Diophantine problems with linear recurrences via the Subspace Theorem}, Integers {\bf 5} (2005), no. 3, A8.

%\bibitem{Fu4} FUCHS, C.: \textit{Polynomial-exponential equations involving multi-recurrences}, Studia Sci. Math. Hungar. {\bf 46} (2009), 377--398.

\bibitem{GoLu} GOMEZ RUIZ, C. A.---LUCA, F.: \textit{Tribonacci Diophantine quadruples}, Glas.~Mat.~ \textbf{50} (2015), no. 1, 17--24.

\bibitem{ISz1} IRMAK, N.---SZALAY, L.: \textit{Diophantine triples and reduced quadruples with the Lucas sequence of recurrence $u_n=Au_{n-1}-u_{n-2}$}, Glas.~Mat.~\textbf{49} (2014), 303--312.

\bibitem{LuSz} LUCA, F.---SZALAY, L.: \textit{Fibonacci Diophantine Triples}, Glas. Mat. \textbf{43(63)} (2008), 253--264.

\bibitem{LuSz1} LUCA, F.---SZALAY, L.: \textit{Lucas Diophantine Triples}, Integers \textbf{9} (2009), 441--457.

%\bibitem{RL1} GOMEZ RUIZ, C. A.---LUCA, F.: \textit{Diophantine quadruples in the sequence of shifted Tribonacci numbers}, Publ.~Math. Debrecen~\textbf{86} (2015), no. 3-4, 473--491.

\bibitem{Mignotte} MIGNOTTE, M.: \textit{Sur les conjugu\'es des nombres de Pisot}, C. R. Acad. Sci. Paris S\'er. I Math. {\bf 298} (1984), no. 2, 21.

%\bibitem{S} SPICKERMAN, W. R.: \textit{Binet's formula for the Tribonacci numbers}, Fibonacci Q. \textbf{20} (1982), 118--120.

\bibitem{Mar} MARTIN, P. A.: \textit{The Galois group of $X^n-X^{n-1}-\cdots-1$}, J. Pure App. Algebra {\bf 190}  (2004), 213--223.

\bibitem{JN} NEUKIRCH, J.: \textit{Algebraic Number Theory}, Springer, {\bf Vol. 322}, (1999).

\bibitem{SzZi} SZALAY, L.---ZIEGLER, V.: \textit{On an $S$-unit variant of Diophantine $m$-tuples}, Publ. Math. Debrecen {\bf 83} (2013),  97--121.

\end{thebibliography}
\end{document}